\documentclass[11pt]{article}
\usepackage{amsmath,amssymb,lmodern,amsthm,breqn,mathrsfs}
\usepackage[margin=30mm]{geometry}

\theoremstyle{definition}
\newtheorem{thm}{Theorem}

\newtheorem{rem}{Remark}

\title{On formal series solutions to 4th-order quadratic homogeneous differential equations and their convergence}
\author{Tatsuya Hosoi}
\date{\empty}

\begin{document}

\maketitle
\begin{abstract}
It is known that all $\tau$ functions of the Painlev\'{e} equations satisfy the fourth-order quadratic differential equation. Among them, for the III, V, and VI equations, it is possible to express the formal series solutions explicitly by using combinatorics. In this paper, we show the convergence of the formal series, including the solutions of more general equations. And by the absolute convergence of $\tau $ series, the convergence of the conformal block function ($c=1$) also follows since it is a partial sum of the $\tau $ series. We also characterized the form of a homogeneous quadratic equation with a series solution similar to the tau functions of the Painlev\'{e} equations. The Painlev\'{e} equations are classified into six types, and it is known that they can be obtained by sequentially degenerating from type VI to type I.
\end{abstract}

\section{Introduction}
For algebraic ordinary differential equations in complex domain, Painlev\'{e} property is the property that movable singularities are only poles. A movable singular point is a singularity whose position depend on the initial value. The Painlev\'{e} equations, which have been found in connection with the Painlev\'{e} analysis and isomonodromic deformations, are second-order nonlinear differential equations that satisfy Painlev\'{e} property. They cannot be reduced to simpler equations such as equations of elliptic functions or linear equations.

Among those six Painlev\'{e} equations, the Painleve III, V, VI equations ($P_{\text{I\hspace{-.1em}I\hspace{-.1em}I}},P_{\text{V}},P_{\text{V\hspace{-.1 em}I}}$) are algebraic differential equations of the second order as follows:
\begin{align}
&P_{\text{V\hspace{-.1em}I}}:& \frac{d^2q}{dt^2}=&\frac{1}{2} \left( \frac{1}{q} +\frac{1}{q-1}+\frac{1}{q-t} \right) \left( \frac{dq}{dt} \right)^2- \left( \frac{1}{t}+ \frac{1}{t-1}+ \frac{1}{q-t} \right) \frac{dq}{dt} &\nonumber \\
&&&+\frac{2q(q-1)(q-t)}{t^2(t-1)^2} \left( \alpha +\frac{\beta t}{q^2}+\frac{\gamma(t-1)}{(q-1)^2} +\frac{\delta t(t-1)}{(q-t)^2}\right),& \\
&P_{\text{V}}:&\frac{d^2q}{dt^2}=&\left( \frac{1}{2q}+ \frac{1}{q-1} \right) \left( \frac{dq}{dt} \right)^2-\frac{1}{t} \frac{dq}{dt}+\frac{(q-1)^2}{t^2}\left( \alpha q+\frac{\beta}{q} \right)+\frac{\gamma q}{t}+\frac{\delta q(q+1)}{q-1}, \\
&P_{\mathrm{I\hspace{-.1em}I\hspace{-.1em}I'}}:&\frac{d^2q}{dt^2}=&\frac{1}{q} \left( \frac{dq}{dt}\right)^2-\frac{1}{t}\frac{dq}{dt}+\frac{q^2(\alpha +\gamma q)}{4t^2}+\frac{\beta}{4t}+\frac{\delta}{4q}.
\end{align}
Here, $\alpha,\beta,\gamma,\delta$ are the parameters of the equations. $P_{\mathrm{I\hspace{-.1em}I\hspace{-.1em}I'}}$ is obtained by transforming the variables to $t_{\mathrm{I\hspace{-.1em}I\hspace{-.1em}I'}}=t_{\mathrm{I\hspace{-.1em}I\hspace{-.1em}I}}^2,q_{\mathrm{I\hspace{-.1em}I\hspace{-.1em}I'}}=t_{\mathrm{I\hspace{-.1em}I\hspace{-.1em}I}}q_{\mathrm{I\hspace{-.1em}I\hspace{-.1em}I}}$ in the following Painlev\'{e} III equation.

\begin{equation}
P_{\mathrm{I\hspace{-.1em}I\hspace{-.1em}I}}:\frac{d^2q}{dt^2}=\frac{1}{q} \left( \frac{dq}{dt}\right)^2-\frac{1}{t}\frac{dq}{dt}+\frac{\alpha q^2+\beta}{t}+\gamma q^3+\frac{\delta}{q}.
\end{equation}

Each equation has fixed singularities, $0,\infty $ for $P_{\mathrm{I\hspace{-.1em}I\hspace{-.1em}I'}}$ and $P_{\mathrm{V}}$, and $0,1,\infty $ for $P_{\text{V\hspace{-.1em}I}}$.

These equations can be written by using Hamiltonian ($\frac{dq}{dt}=\frac{\partial H_{\text{J}}}{\partial p},\frac{dp}{dt}=-\frac{\partial H_{\text{J}}}{\partial q},\text{J}=\text{V\hspace{-.1em}I},\text{V},\text{I\hspace{-.1em}I\hspace{-.1em}I}'_{1,2,3}$).
\begin{align}
t(t-1)H_{\text{V\hspace{-.1em}I}}=&q(q-1)(q-t)p\left( p-\frac{2\theta_0}{q}-\frac{2\theta_1}{q-1}-\frac{2\theta_t-1}{q-t} \right) \nonumber \\
&+(\theta_0+\theta_t+\theta_1+\theta_\infty)(\theta_0+\theta_t+\theta_1-\theta_\infty-1)q ,\\
tH_{\text{V}}=&(q-1)(pq-2\theta_t)(pq-q+2\theta_*)-tpq+\left( (\theta_*+\theta_t)^2-\theta_0^2 \right)q \nonumber \\
&+\left( \theta_t-\frac{\theta_*}{2}\right)t-2\left( \theta_t+\frac{\theta_*}{2}\right)^2 ,\\
tH_{\text{I\hspace{-.1em}I\hspace{-.1em}I}'_1}=&(pq+\theta_*)^2+tp-\theta_\star q-\frac{q^2}{4} ,\\
tH_{\text{I\hspace{-.1em}I\hspace{-.1em}I}'_2}=&(pq+\theta_*)^2+tp-q ,\\
tH_{\text{I\hspace{-.1em}I\hspace{-.1em}I}'_3}=&p^2 q^2-q-\frac{t}{q}.
\end{align}
Then, $\alpha,\beta,\gamma,\delta$ are written by $\theta $'s and $\kappa $'s, and with scale transformations of the variables, the Hamiltonian systems encompasses all of $\alpha,\beta,\gamma,\delta$ above.
The notation is based on \cite{gamayun}.

The $\tau$ functions of the Painlev\'{e} equations are then defined as follows:
\begin{align}
t(t-1)\frac{d}{dt}\log \tau_{\mathrm{V\hspace{-.1em}I}}(t)=&t(t-1)H_{\mathrm{V\hspace{-.1em}I}}-q(q-1)p+(\theta_0+\theta_t+\theta_1+\theta_\infty)q \nonumber \\
&-(\theta_1^2+2\theta_0 \theta_1-\theta_\infty^2)t-( \theta_0 +\theta_t)^2 ,\\
t\frac{d}{dt}\log \tau_{\mathrm{J}}(t)=&tH_{\mathrm{J}}, ~~~ \mathrm{J}=\mathrm{V}, \mathrm{I\hspace{-.1em}I\hspace{-.1em}I}'_{1,2,3}.
\end{align}

These $\tau$ functions are holomorphic except for the fixed singularities in the equations ($\left\{ 0,\infty \right\} $ for $\tau_{\mathrm{I\hspace{-.1em}I\hspace{-.1em}I'}},\tau_{\mathrm{V}}$, and $\left\{ 0,1,\infty \right\} $ for $P_{\mathrm{V\hspace{-.1em}I}}$), and does not even have a movable  pole. Just as the elliptic function was expressed as a ratio of theta functions which are holomorphic, the solution to the Painlev\'{e} equation can be expressed as a rational function of the $\tau$ function and its derivative.
For these $\tau$ functions, explicit formal series solutions of the following form were conjectured by Gamayun, Iorgov, Lisovyy \cite{gamayun}.
\begin{equation}
\tau (t)=\sum_{m\in \mathbb{Z},n\in \mathbb{Z}_{\ge 0}} b_{m,n}(\sigma) t^{(\sigma +m)^2+n} .
\end{equation}
Here, $\sigma$ is an initial value-dependent complex parameter. After that, this expression was proven in multiple ways \cite{bershtein, gavrylenko, iorgov}. For the Painlev\'{e} equations of type I, I\hspace{-.1em}I, I\hspace{-.1em}V, the $\tau$ functions are holomorphic at $t=0$, but is conjectured to have more complicated presentations in their asymptotic expansions at $\infty$ \cite{bonelli,lisovyy}.

Here in pariticular, $\tau_{\mathrm{V\hspace{-.1em}I}}$ is written as follows \cite{bershtein,gavrylenko,iorgov}:
\begin{equation}
\tau_{\mathrm{V\hspace{-.1em}I}}(t)=\text{const} \cdot t^{-\theta_0^2-\theta_t^2} \cdot \sum_{n \in \mathbb{Z}} e^{in\eta}\mathscr{B} \left( \vec{\theta},\sigma +n;t \right). \\
\end{equation}
We get similar expressions for type III and V from the degenerartion of  type V\hspace{-.1em}I \cite{gamayun}. Here $\mathscr{B}(\vec{\theta},\sigma;t)$ is the sum over the all pairs of the Young diagrams, defined as
\begin{align}
\mathscr{B} \left( \vec{\theta},\sigma ;t\right)=&\mathscr{N}^{\theta_1}_{\theta_\infty ,\sigma} \mathscr{N}^{\theta_t}_{\sigma ,\theta_0}t^{\sigma^2}(1-t)^{2 \theta_t \theta_1} \sum_{\lambda ,\mu \in \mathbb{Y}} \mathscr{B}_{\lambda ,\mu}\left( \vec{\theta} ,\sigma \right) t^{\left| \lambda \right|+\left| \mu \right|}, \nonumber \\
\mathscr{B}_{\lambda ,\mu} \left( \vec{\theta} , \sigma \right)=&\prod_{(i,j)\in \lambda}\frac{\left( (\theta_t-\sigma+i-j)^2-\theta_0^2 \right) \left( (\theta_1-\sigma+i-j)^2-\theta_\infty^2 \right)}{h_\lambda^2(i,j)\left( \lambda_j^\prime -i+\mu_i-j+1+2 \sigma \right)^2}, \nonumber \\
&\times \prod_{(i,j)\in \mu}\frac{\left( (\theta_t-\sigma+i-j)^2-\theta_0^2 \right) \left( (\theta_1-\sigma+i-j)^2-\theta_\infty^2 \right)}{h_\mu^2(i,j)\left( \mu_j^\prime -i+\lambda_i-j+1+2 \sigma \right)^2}, \nonumber \\
\mathscr{N}^{\theta_2}_{\theta_3,\theta_1}=&\frac{\prod_{\varepsilon =\pm} G(1+\theta_3+\varepsilon (\theta_1+\theta_2))G(1-\theta_3+\varepsilon (\theta_1-\theta_2))}{G(1-2 \theta_1)G(1-2 \theta_2)G(1+2 \theta_3)}. \nonumber 
\end{align}
Here $\sigma \notin \mathbb{Z}/2,\eta$ are arbitrary complex parameters, and $G(z)$ denotes the Barnes $G$ function. For this series solution, the recurrence relation $G(1)=1,G(z+1)=\mathrm{\Gamma}(z)G(z)$ is important. And, $\lambda^{\prime}$ is the conjugate of the Young diagram $\lambda$ and $h_\lambda (i,j)$ is the hook length of the cell $(i,j)$ of $\lambda$. For $\tau_{\mathrm{I\hspace{-.1em}I\hspace{-.1em}I}'_3}$, it is directly proved from the concrete expression of this series that its formal series converges over the entire universal covering of $\mathbb{C}\setminus \{ 0\}$ \cite{its}. But the same method could not be applied to other $\tau$ functions and the convergence problem remained unsolved.

On the other hand, these $\tau$ functions are solutions to a fourth-order homogeneous quadratic differential equation. In this paper, we show that the $\tau$ functions have a convergence region by constructing a dominating series, and furthermore, we show in what region the convergence is guaranteed.

For $P_{\mathrm{I\hspace{-.1em}I\hspace{-.1em}I}}$ and $P_{\mathrm{V}}$, the $\tau$ functions are the solutions to the following equation
\begin{align}
\tau_{\mathrm{V}}:~& \mathcal{D}^4_{\log t}f\cdot f +(4\sigma_2 +1-t^2)\mathcal{D}^2_{\log t}f\cdot f-4\mathcal{D}^2_{\log t}(\delta f)\cdot f-4\sigma_3 tf\cdot f=0, \\
\tau_{\mathrm{I\hspace{-.1em}I\hspace{-.1em}I}'_1}:~& \mathcal{D}^4_{\log t}f\cdot f+\mathcal{D}^2_{\log t}f\cdot f-4\mathcal{D}^2_{\log t}(\delta f)\cdot f +4\theta_\star \theta_* t f^2+t^2 f^2=0,\\
\tau_{\mathrm{I\hspace{-.1em}I\hspace{-.1em}I}'_2}:~& \mathcal{D}^4_{\log t}f\cdot f+\mathcal{D}^2_{\log t}f\cdot f-4\mathcal{D}^2_{\log t}(\delta f)\cdot f+4\theta_* t f^2 =0,\\
\tau_{\mathrm{I\hspace{-.1em}I\hspace{-.1em}I}'_3}:~& \mathcal{D}^4_{\log t}f\cdot f+\mathcal{D}^2_{\log t}f\cdot f-4\mathcal{D}^2_{\log t}(\delta f)\cdot f+4t f^2=0.
\end{align}
Here, $\sigma_2$ and $\sigma_3$ are 2nd and 3rd order elementary symmetric expressions of $v_j$ ($j=1,2,3,4$) : $v_1=\frac{1}{2} \theta_* +\theta_0$, $v_2=\frac{1}{2} \theta_* -\theta_0$, $v_3=-\frac{1}{2} \theta_* +\theta_t$, $v_4-\frac{1}{2} \theta_* -\theta_t$ for $\tau_{\mathrm{V}}$. The differential operator $\delta$ is defined as $\delta =t\frac{d}{dt}$. Also, $\mathcal{D}_{\log t}$ is the Hirota derivative with respect to $\delta$ and is defined as follows. These are bilinear operators.
\begin{equation}
\mathcal{D}^N_{\log t} f\cdot g:=\sum_{i=0}^N (-1)^i \binom{N}{i} \delta^{N-i}f \cdot \delta^ig.
\end{equation}
In particular, the following holds.
\begin{align}
\mathcal{D}^2_{\log t} f\cdot g=&\delta^2f \cdot g-2\delta f\cdot g+f\cdot \delta^2 g ,\\
\mathcal{D}^4_{\log t} f\cdot g=&\delta^4 f\cdot g-4\delta^3f \cdot \delta g+6\delta^2f \cdot \delta^2g-4\delta f\cdot \delta^3g+f\cdot \delta^4g.
\end{align}

 And for $P_{\mathrm{V\hspace{-.1em}I}}$, $f=t^{\frac{\theta_0^2+\theta_t^2-\theta_1^2-\theta_\infty^2}{2}}(1-t)^{\frac{\theta_t^2+\theta_1^2-\theta_0^2-\theta_\infty^2}{2}}\tau_{\mathrm{V\hspace{-.1em}I}}$ is the solution to the following equation
\begin{gather}
\tau_{\mathrm{V\hspace{-.1em}I}}: \mathcal{D}_{\bar{D}}^4f\cdot f +(2t(t-1)+1-e_1)\mathcal{D}_{\bar{D}}^2f\cdot f-4(2t-1)\mathcal{D}_{\bar{D}}^2 (\bar{D}f)\cdot f \nonumber \\
-t(t-1)(2\sigma_4(2t-1)+e_2)f\cdot f +4t(t-1)(\bar{D}f)\cdot (\bar{D}f)=0.
\end{gather}
Similarly, $\mathcal{D}_{\bar{D}}$ is the Hirota derivative with respect to $\bar{D}=t(t-1)\frac{d}{dt}$.

For these quadratic differential equations for $\tau$, the lowest degree parts in $t$ (the differential operator $\frac{d}{dt}$ is counted as degree $-1$) have the common form $\mathcal{D}^4_{\log t}f\cdot f+(1+\alpha)\mathcal{D}^2_{\log t}f\cdot f-4\mathcal{D}^2_{\log t}(\delta f)\cdot f $. This is also true for $\tau_{\mathrm{V\hspace{-.1em}I}}$ when its equation is rewritten in $\mathcal{D}_{\log t}$. Since $\alpha $ in the equations are canceled by gauge transformation of the solution $f \mapsto t^{\frac{\alpha}{4}} f$, we can consider $\mathcal{D}^4_{\log t}f\cdot f+\mathcal{D}^2_{\log t}f\cdot f-4\mathcal{D}^2_{\log t}(\delta f)\cdot f $. For a more general form of the homogeneous quadratic equations involving the equations satisfied by these $\tau$ functions, we show the convergence of their solutions.

\section{Fourth order homogeneous quadratic equations}
 The $\tau$ functions of the Painlev\'{e} I\hspace{-.1em}I\hspace{-.1em}I, V, V\hspace{-.1em}I equations satisfy the homogeneous quadratic differential equations and there exists formal series solutions of the following form \cite{gamayun}.
\begin{equation}
\tau (t)=\sum_{m\in \mathbb{Z},n\in \mathbb{Z}_{\ge 0}} b_{m,n}(\sigma) t^{(\sigma +m)^2+n} \label{eqe}
\end{equation}

This is also a special case of of a series of the following form:
\begin{equation}
f(t)=t^{\sigma^2}\sum_{\substack{(m,n)\in {\mathbb{Z}_{\ge 0}}^2 \\ \text{or} \ (m,n)\in \left( \frac{1}{2} +\mathbb{Z}_{\ge 0} \right)^2}} a_{m,n}(\sigma)t^{(2\sigma +1)m+(-2\sigma +1)n}. \label{eqa}
\end{equation}
For a homogeneous quadratic equation with a series solution of this form, the following holds.

\begin{thm}
Let $N_0$ be a positive integer. Consider a quadratic differential equation for $f$:
\begin{align}
\sum_{N=1}^{N_0} t^N \sum_{0 \le K_1 \le K_2< \infty}\alpha_{N,K_1,K_2}t^{K_1+K_2} \frac{d^{K_1}f}{dt^{K_1}} \cdot \frac{d^{K_2}f}{dt^{K_2}} \nonumber \\
+\sum_{0 \le K_1 \le K_2 \le 4}\alpha_{0,K_1,K_2}t^{K_1+K_2} \frac{d^{K_1}f}{dt^{K_1}} \cdot \frac{d^{K_2}f}{dt^{K_2}}=0. \label{eqb}
\end{align}
A necessary and sufficient condition for (\ref{eqb}) to have a formal series solution satisfying $a_{0,0},a_{1,0},a_{0,1}\ne 0$ of the form:
\begin{equation}
f(t)=t^{\sigma^2}\sum_{\substack{(m,n)\in {\mathbb{Z}_{\ge 0}}^2 \\ \text{or} \ (m,n)\in \left( \frac{1}{2} +\mathbb{Z}_{\ge 0} \right)^2}} a_{m,n}(\sigma)t^{(2\sigma +1)m+(-2\sigma +1)n}
\end{equation}
with $\sigma \notin \mathbb{Q}$ as an arbitrary parameter is, that the lowest-degree part of the equation for $t$ $\sum_{0 \le K_1 \le K_2 \le 4}\alpha_{0,K_1,K_2}t^{K_1+K_2} \frac{d^{K_1}f}{dt^{K_1}} \cdot \frac{d^{K_2}f}{dt^{K_2}}$ be a constant multiple of
\begin{equation}
2t^4\left\{ f \cdot \frac{d^4f}{dt^4}-4\frac{df}{dt} \cdot \frac{d^3f}{dt^3}+3\left( \frac{d^2f}{dt^2} \right)^2 \right\}+8t^3\left( f \cdot \frac{d^3t}{dt^3} -\frac{df}{dt} \cdot \frac{d^2f}{dt^2} \right)+4t^2f \cdot \frac{d^2f}{dt^2}. \label{eq+}
\end{equation}
In particular, quadratic differential equations of the third order and below do not have solutions of the form described above.
\end{thm}

\begin{proof}

First, since the differential operator $\delta=t\frac{d}{dt}$ does not change the degree with respect to $t$, we rewrite (\ref{eq+}) into an equation for $\delta$ to prove it. That is, instead of (\ref{eq+}), we consider 
\begin{equation}
2 f\cdot \delta^4f-8\delta f\cdot \delta^3f+6(\delta^2f)^2+2f\cdot \delta^2f-2(\delta f)^2-4f\cdot \delta^3f+4\delta f\cdot \delta^2 f.
\end{equation}
This can be expressed by using Hirota derivative as follows.
\begin{equation}
\mathcal{D}_{\log t}^4f\cdot f+\mathcal{D}_{\log t}^2f\cdot f-4\mathcal{D}_{\log t}(\delta f)\cdot f.
\end{equation}

(Necessity) Substituting (\ref{eqa}) for (\ref{eqb}), the coefficient of $t^{2 \sigma ^2}$ is
\begin{equation}
\sum_{0\le K_1\le K_2\le 4}\alpha_{0,K_1,K_2}\sigma^{2(K_1+K_2)}a_{0,0}^2 .
\end{equation}
Since this must be $0$ regardless of $\sigma$, considering the coefficient of $\sigma^{2K_0}$, for $K_0\ (=K_1+K_2)$ which is $0\le K_0\le 8$,
\begin{equation}
\sum_{\substack{0\le K_1\le K_2\le 4\le K_1+K_2=K_0}}\alpha_{0,K_1,K_2}=0,
\end{equation}
is satisfied. When written down, 
\begin{gather}
\alpha_{0,0,0}=0,~ \alpha_{0,0,1}=0,~ \alpha_{0,3.4}=0,~ \alpha_{0,4,4}=0, \nonumber \\
\alpha_{0,0,2}+\alpha_{0,1,1}=0,~ \alpha_{0,0,3}+\alpha_{0,1,2}=0,~ \alpha_{0,1,4}+\alpha_{0,2,3}=0,~ \alpha_{0,2,4}+\alpha_{0,3,3}=0, \nonumber \\
\alpha_{0,0,4}+\alpha_{0,1,3}+\alpha_{0,2,2}=0. \label{eqc}
\end{gather}

And also the coefficient of $t^{2 \sigma^2+2 \sigma +1}$ is,
\begin{equation}
\sum_{0\le K_1\le K_2\le 4} \alpha_{0,K_1,K_2}\left( \sigma^{2K_1}(\sigma +1)^{2K_2}+(\sigma+1)^{2K_1}\sigma^{2K_2}\right) a_{0,0}a_{1,0},
\end{equation}
In combination with (\ref{eqc}), 
\begin{gather}
\sigma^4 (\sigma+1)^4(2 \sigma +1)^2 \alpha_{0,2,4}+\sigma^2(\sigma+1)^2(2\sigma +1)^2(2 \sigma^2+2\sigma+1)\alpha_{0,1,4} \nonumber \\
+\left( \sigma^8+(\sigma+1)^8\right) \alpha_{0,0,4}+\sigma^2 (\sigma+1)^2 \left( \sigma^4 +(\sigma+1)^4\right) \alpha_{0,1,3}-2\sigma^4(\sigma+1)^4 (\alpha_{0,0,4}+\alpha_{0,1,3}) \nonumber \\
+(2 \sigma+1)^2(2 \sigma^2+ 2\sigma+1)\alpha_{0,0,3}+(2\sigma+1)^2\alpha_{0,0,2} =0,
\end{gather}
Since this equation holds regardless of $\sigma$, the claim follows by comparing the coefficients of the terms, starting from the higher order ones with respect to $\sigma$.

(Sufficiency) If the latter part of (\ref{eqb}), which is the lowest degree part of $t$, is of the form (\ref{eq+}), then substituting (\ref{eqa}), the coefficients of $t^{2\sigma^2+(2\sigma+1)m+(-2\sigma+1)n}$ are functions of $a_{k,l}$ such that $0 \le k \le m,0\le l \le n$ and its only term containing $a_{m,n}$ is
\begin{equation}
4\left( (2\sigma +1)m+(-2\sigma+1)n\right)^2\left\{ \left( (2\sigma +1)m+(-2\sigma+1)n-1\right)^2 -4\sigma^2 \right\}a_{0,0} a_{m,n}. \label{eqt}
\end{equation}
Only when $(m,n)=(0,0),(0,1),(1,0)$, this term becomes zero. Also, when $(m,n)=(0,0),(0,1),(1,0)$, the coefficient of $t^{2\sigma^2+(2\sigma+1)m+(-2\sigma+1)n}$ is exactly (\ref{eqt}), which is $0$. Therefore, if $a_{0,0}, a_{0,1}, a_{1,0}$ are determined, the remaining coefficients $a_{m,n}$ are uniquely determined, and the series is the formal solution of the equation.
\end{proof}

\begin{rem}
Stronger than the theorem, it is observed that the solutions of some equations of the form (\ref{eqa}), including the Painlev\'{e} equations, actually take the form (\ref{eqe}). That is, when the solution is expressed as:
\begin{equation}
f(t)=t^{\sigma^2} \sum_{x\in \mathbb{Z},y\in \mathbb{Z}_{\ge 0}} c_{x,y} t^{(2\sigma x+y)},
\end{equation}
we showed that $c_{x,y}=0$ for $x,y$ such that $y < \left| x \right|$, but actually $c_{x,y}=0$ for $x,y$ such that $y<x^2$ in some equations.
\end{rem}

Next, we prove the convergence of the formal solution.
\begin{thm}
Let $N_0$ be a positive integer. Consider a quadratic equation of the following form
\begin{equation}
\mathcal{D}_{\log t}^4f\cdot f+\mathcal{D}_{\log t}^2f\cdot f-4\mathcal{D}^2_{\log t}(\delta f)\cdot f +\sum_{N=1}^{N_0} t^N \sum_{K=0}^{4} \sum_{I=0}^{\left[ \frac{1}{2} K \right]}\alpha_{N,K,I}\delta^I f \delta^{K-I}f=0, \label{eq1}
\end{equation}
And consider its formal series solution of the form
\begin{equation}
f(t)=t^{\sigma^2}\sum_{\substack{(m,n)\in {\mathbb{Z}_{\ge 0}}^2 \\ \text{or} \ (m,n)\in \left( \frac{1}{2} +\mathbb{Z}_{\ge 0} \right)^2}} a_{m,n}t^{(2\sigma +1)m+(-2\sigma +1)n}, \label{eq2}
\end{equation}
When $\sigma \notin \mathbb{Q},\left| \mathrm{Re} \ \sigma \right| <\frac{1}{2}$, the formal series has a convergence domain in $\mathbb{C}\backslash\{ 0\}$ that contains a neighborhood of $0$.
\end{thm}

\begin{proof}
As we saw in the proof of Theorem 1, when $a_{0,0},a_{1,0},a_{0,1}$ are defined, the remaining coefficients are uniquely determined starting from the one with small $m+n$. We clarify the recurrence formula of the coefficients and construct a simpler dominant series.

Taking note of
\begin{align}
\mathcal{D}_{\log t}^4 (a t^\alpha +bt^\beta) \cdot (at^\alpha +bt^\beta)&=2(\alpha -\beta)^4 abt^{\alpha+\beta}, \\
\mathcal{D}_{\log t}^2 \delta (at^\alpha +bt^\beta) \cdot (at^\alpha +bt^\beta)&=(\alpha -\beta)^2 (\alpha+ \beta) abt^{\alpha+\beta}, \\
\mathcal{D}_{\log t}^2 (at^\alpha +bt^\beta) \cdot (at^\alpha +bt^\beta)&=2(\alpha -\beta)^2 abt^{\alpha+\beta},
\end{align}
when substituting (\ref{eq2}) into the left side of (\ref{eq1}), we obtain the coefficients of $t^{2 \sigma^2+(2\sigma +1)m+(-2\sigma +1)n}$ as follows.

\begin{gather}
\sum_{0\le k \le m,0\le l \le n}\rho_{k,l}^2\left\{ (\rho_{k,l}-1)^2-4 \sigma^2 -4 \left\{ (2 \sigma +1)k+(-2\sigma +1)l\right\} \right\}a_{k,l}a_{m-k,n-l} \nonumber \\
+ \sum_{N=1}^{N_0} \sum_{0\le k \le m-\frac{1}{2}N,0\le l\le n-\frac{1}{2}N} \sum_{K=0}^{4} \sum_{I=0}^{\left[ \frac{1}{2} K \right]}  \alpha_{N,K,I}\left( \sigma^2 +(2 \sigma +1)k +(-2 \sigma +1)l \right)^I \label{eq3} \\
\cdot \left( \sigma^2 +(2 \sigma+1)(m-\frac{1}{2}N-k)+(-2 \sigma +1)(n-\frac{1}{2}N-l)\right)^{K-I} a_{k,l}a_{m-\frac{1}{2}N-k,n-\frac{1}{2}N-l}. \nonumber
\end{gather}
We set $\rho_{k,l}:=(2\sigma +1)(m-2k)+(-2 \sigma +1)(n-2l)$ here. Note that for each sum $(k,l)\in {\mathbb{Z}_{\ge 0}}^2 $ or $(k,l)\in \left( \frac{1}{2} +\mathbb{Z}_{\ge 0} \right)^2$. From the condition that (\ref{eq3}) is $0$, $a_{m,n}$ is determined. Since we assumed $\sigma \notin \mathbb{Q}$, $t^{2 \sigma^2+(2\sigma +1)m+(-2\sigma +1)n}$ for different $m,n$ will never match.

When $m+n\ge 2$ or $(m,n)=(\frac{1}{2},\frac{1}{2})$, there exist constants $R,A_N>0\ (N\in \mathbb{Z}_{>0})$ independent of $m,n$ and following estimate holds.
\begin{align}
\left| a_{m,n} \right| \le & \frac{R}{\left| a_{0,0}\right| }\sum_{\frac{1}{2}\le k \le m-\frac{1}{2},\frac{1}{2}\le l-1 \le n-\frac{1}{2}}\left| a_{k,l}a_{m-k,n-l}\right| \nonumber \\
&+\frac{1}{\left| a_{0,0} \right|}\sum_{N=1}^\infty A_N \sum_{0\le k \le m-\frac{1}{2}N,0\le l\le n-\frac{1}{2}N} \left| a_{k,l}a_{m-\frac{1}{2}N-k,n-\frac{1}{2}N-l}\right| \label{eq4}
\end{align}

Since $\left| \rho_{k,l}\right| \le (2\left| \sigma \right|+1)(m+n)$ and $ \left| (2 \sigma +1)k+(-2\sigma +1)l \right| \le (2\left| \sigma \right|+1)(m+n)$, $R$ can be taken as
\begin{equation}
R=\frac{2(2\left| \sigma \right|+1)^2\{ (2\left| \sigma \right|+1)^2+4|\sigma|^2+2(2\left| \sigma \right|+1)\} }{L^4}.
\end{equation}

Also, when $m+n \ge 1$, the following holds.
\begin{gather}
\alpha_{N,K,I}\left( \sigma^2 +(2 \sigma +1)k +(-2 \sigma +1)l \right)^I \left( \sigma^2 +(2 \sigma+1)(m-\frac{1}{2}N-k)+(-2 \sigma +1)(n-\frac{1}{2}N-l)\right)^{K-I} \nonumber \\
\le \left| \alpha_{N,K,I} \right| \left( |\sigma |^2 +(2 |\sigma |+1)(m+n)\right)^4 \le \left| \alpha_{N,K,I} \right| \left( |\sigma| +1\right)^2 (m+n)^4.
\end{gather}

Therefore, setting $A_N:=\frac{2 \left( \left| \sigma \right| +1\right)^2}{L^4} \sum_{K=0}^{4} \sum_{I=0}^{\left[ \frac{1}{2} K \right]} \left| \alpha_{N,K,I} \right|$, then (\ref{eq4}) holds when $m+n\ge 2$.

Then, among the two solutions $g(t)=t^{\sigma^2} \sum_{m,n\ge 0}b_{m,n}t^{(2\sigma +1)m+(-2\sigma +1)n}$ of the following equation, the one with $b_{0,0}=-\frac{|a_{0,0}|}{2R}$ is the dominant series.
\begin{gather}
\left( 1+\sum_{N=1}^{N_0} A_N t^N \right) g^2+\sum_{N=1}^{N_0} A_Nt^N \frac{|a_{0,0}|}{R} t^{\sigma^2}g\nonumber \\
=\frac{|a_{0,0}|}{2R}\left( \frac{|a_{0,0}|}{2R}-2|a_{1,0}|t^{2\sigma +1}-2|a_{0,1}|t^{-2\sigma +1}-\sum_{N=1}^{N_0} A_N \left(\frac{|a_{0,0}|}{R}+2|a_{0,0}|\right) t^N\right)t^{2\sigma^2},
\end{gather}
The solution $g(t)$ is expressed as follows.
\begin{align}
g(t)&=-t^{\sigma^2} \frac{\sum_{N=1}^{N_0} A_Nt^N \frac{|a_{0,0}|}{R}}{2 \left( 1+\sum_{N=1}^{N_0} A_N t^N\right)}+t^{\sigma^2}\sqrt{\left\{ \frac{\sum_{N=1}^{N_0} A_Nt^N \frac{|a_{0,0}|}{R}}{2 \left( 1+\sum_{N=1}^{N_0} A_N t^N\right)}\right\}^2 +\gamma (t)}, \nonumber \\
\gamma (t)&=\frac{|a_{0,0}|}{2R}\left( \frac{|a_{0,0}|}{2R}-2|a_{1,0}|t^{2\sigma +1}-2|a_{0,1}|t^{-2\sigma +1}-\sum_{N=1}^{N_0} A_N \left(\frac{|a_{0,0}|}{R}+2|a_{0,0}|\right) t^N\right).
\end{align}
For $b_{0,0}$, it is smaller than that of the original series $f(t)$, but all the recurrence formulas at $m+n \ge 1$ are more increasing than (\ref{eq4}), and $b_{m,n}\ge \left| a_{m,n} \right|$ holds. Therefore, $g(t)$ is the dominant series of $f(t)$. Specifically, we have
\begin{equation}
b_{1,0}=|a_{1,0}|,~b_{0,1}=|a_{0,1}|,
\end{equation}

If $(m,n) \ne (1,0),(0,1)$ and $m \ne n$, then
\begin{align}
b_{m,n} = & \frac{R}{\left| a_{0,0}\right| }\sum_{\frac{1}{2}\le k \le m-\frac{1}{2},\frac{1}{2}\le l-1 \le n-\frac{1}{2}} b_{k,l}b_{m-k,n-l} \nonumber \\
&+\frac{1}{\left| a_{0,0} \right|}\sum_{N=1}^{N_0} A_N \sum_{0\le k \le m-\frac{1}{2}N,0\le l\le n-\frac{1}{2}N}  b_{k,l}b_{m-\frac{1}{2}N-k,n-\frac{1}{2}N-l}
\end{align}
And if $(m,n)\ne (0,0)$ and $m=n$, then
\begin{align}
b_{m,n} = & \frac{R}{\left| a_{0,0}\right| }\sum_{\frac{1}{2}\le k \le m-\frac{1}{2},\frac{1}{2}\le l-1 \le n-\frac{1}{2}} b_{k,l}b_{m-k,n-l} \nonumber \\
&+\frac{1}{\left| a_{0,0} \right|}\sum_{N=1}^{N_0} A_N \sum_{\frac{1}{2}\le k \le m-\frac{1}{2}N,\frac{1}{2}\le l\le n-\frac{1}{2}N}  b_{k,l}b_{m-\frac{1}{2}N-k,n-\frac{1}{2}N-l} \nonumber \\
&+A_{m+n} \left( \frac{|a_{0,0}|}{2R}+|a_{0,0}| \right)
\end{align}
All of the $b_{k,l}$ are larger than $\left| a_{k,l} \right|$. And compared to (\ref{eq4}), $b_{m,n}\ge \left| a_{m,n} \right|$ holds.
\end{proof}

\begin{rem}
The convergence of the series follows if we use the results of \cite{gontsov}, but here is a simple proof in a form that shows the domain of convergence.
The solution series to the equation of Theorem $2$ converge in the region where the series expansion around $0$ in the following converges.
\begin{equation}
\sqrt{\left\{ \frac{\sum_{N=1}^{N_0} A_N\left| t\right|^N \frac{|a_{0,0}|}{R}}{2 \left( 1+\sum_{N=1}^{N_0} A_N \left| t\right|^N\right)}\right\}^2+\tilde{\gamma}(t)}. \label{eqsqrt}
\end{equation}
Here, we set
\begin{equation}
\tilde{\gamma}(t)=\frac{|a_{0,0}|}{2R}\left( \frac{|a_{0,0}|}{2R}-2|a_{1,0}t^{2\sigma +1}|-2|a_{0,1}t^{-2\sigma +1}|-\sum_{N=1}^{N_0} A_N t^N\left(\frac{|a_{0,0}|}{R}+2|a_{0,0}|\right) t^N\right) .
\end{equation}
The range of convergence is not empty because (\ref{eqsqrt}) does not become $0$ when $t=0$ is substituted.

Although $2\sigma +1$ and $2\sigma -1$ appear in the power of $t$, the region of convergence includes the area around $0$ since $\left| \mathrm{Re} \ \sigma \right| <\frac{1}{2}$.
\end{rem}

\begin{rem}
Although we assumed $\left| \mathrm{Re} \ \sigma \right| <\frac{1}{2}$ for the complex parameter $\sigma$ in Theorem $2$, if we know that the formal solution of the equation has the form (\ref{eqe}), we can shift $m$ by an integer and instead use the following notation to satisfy the assumptions of Theorem $2$ ($\left| \mathrm{Re} \ \sigma -m_0 \right| <\frac{1}{2}$). Here, $m_0 $ is the integer closest to $\sigma$.
\begin{equation}
\tau (t)=\sum_{m\in \mathbb{Z},n\in \mathbb{Z}_{\ge 0}} b_{m-m_0,n}(\sigma) t^{(\sigma-m_0 +m)^2+n} 
\end{equation}
Therefore, convergence of formal solutions of the form (\ref{eqe}) is proved for any $\sigma$ satisfying $\left| \mathrm{Re} \ \sigma \right| \notin \frac{1}{2}+\mathbb{Z}$.
\end{rem}

\begin{rem}
 By the absolute convergence of the $\tau $ series, the convergence of the conformal block function $\mathscr{B} \left( \vec{\theta},\sigma ;t\right)$ also follows since it is a partial sum of the $\tau $ series.
\end{rem}

\end{document}